\RequirePackage{plautopatch}
\documentclass[12pt]{amsart}
\usepackage{amsmath,amsthm,amssymb,mathtools,latexsym,amsfonts,longtable,bm,thmtools,comment,etoolbox}
\usepackage[margin=20mm]{geometry}
\usepackage[dvipsnames]{xcolor}
\usepackage[
  final,
  colorlinks=true,
  linkcolor=MidnightBlue,
  citecolor=ForestGreen,
  urlcolor=BrickRed
]{hyperref}
\usepackage[nameinlink,noabbrev]{cleveref}
\hypersetup{
  pdftitle={The minimal obstruction modulus for quadratic forms},
  pdfauthor={Naoki Ochi}
}

\theoremstyle{definition}
\newtheorem{dfn}{Definition}[section]
\newtheorem{prop}[dfn]{Proposition}
\newtheorem{lem}[dfn]{Lemma}
\newtheorem{thm}[dfn]{Theorem}
\newtheorem{col}[dfn]{Corollary}
\newtheorem{rem}[dfn]{Remark}

\newtheorem{ex}[dfn]{Example}

\newtheorem*{clm}{Claim}

\crefname{dfn}{Definition}{Definitions}
\crefname{prop}{Proposition}{Propositions}
\crefname{lem}{Lemma}{Lemmas}
\crefname{thm}{Theorem}{Theorems}
\crefname{col}{Corollary}{Corollaries}
\crefname{rem}{Remark}{Remarks}
\crefname{conj}{Conjecture}{Conjectures}
\crefname{fact}{Fact}{Facts}
\crefname{clm}{Claim}{Claims}
\crefname{ex}{Example}{Examples}
\crefname{ques}{Question}{Questions}

\newcommand{\R}{\mathbb{R}}

\newcommand{\Z}{\mathbb{Z}}

\newcommand{\legendre}[2]{\left(\frac{#1}{#2}\right)}

\title[The minimal obstruction modulus for quadratic forms]{The minimal obstruction modulus \\for quadratic forms}
\author{Naoki Ochi}
\address{Graduate School of Mathematics, Nagoya University, Nagoya, Japan.}
\email{ochi.naoki.d1@s.mail.nagoya-u.ac.jp or o.naoki.math@outlook.jp}
\date{\today}

\begin{document}

\begin{abstract}
Let $Q = ax^2+bxy+cy^2$ be a primitive positive definite integral binary quadratic form with
discriminant $\Delta = b^2-4ac$.
It is known that $Q$ admits a \emph{local obstruction}; that is,
there exist $k,l \in \Z$ such that $Q \not\equiv l \pmod k$.
We study the \emph{minimal obstruction modulus}
$\kappa_Q \coloneq \min \{k \in \Z_{\geq 1} \mid \text{there exists } l \text{ such that } Q \not\equiv l \pmod k \}$,
and we determine $\kappa_Q$ completely, treating the cases $\Delta \equiv 0 \pmod 4$ and $\Delta \equiv 1 \pmod 4$ separately.
We also determine the analogous invariants for primitive ternary diagonal forms.
\end{abstract}

\maketitle

\section{Introduction}
The study of integers represented by binary quadratic forms is a classical topic in number theory, dating back to Fermat's observation that an odd prime $p$ satisfies $p = x^2 + y^2$ if and only if $p \equiv 1 \pmod{4}$. A systematic theory was established by Gauss in his \emph{Disquisitiones Arithmeticae}~\cite{gauss1801}, where he introduced equivalence, reduction, and composition of forms, and related representation problems to the arithmetic of class groups.
 
Subsequent work by Dirichlet, Dedekind, and others clarified the connection between binary quadratic forms and ideal classes in quadratic number fields. In this setting, determining which integers are represented by a given form is closely tied to the structure of the corresponding class group and its decomposition into genera. For a thorough account of this history, we refer to~\cite{coxprime}.
 
A key feature of the theory is the role of local conditions. If an integer is represented by a binary quadratic form over the integers, then it must be represented over all completions, in particular over $\Z_p$ for every prime $p$ and over $\R$. This viewpoint fits naturally into the framework of the Hasse--Minkowski principle, which asserts that a quadratic form represents a number over a number field if and only if it does so locally everywhere; see~\cite{jones1950} for a detailed treatment. Together with the local theory of quadratic forms, this implies that every positive definite integral quadratic form in at most three variables fails to represent some $p$-adic integer at a finite prime $p$. Equivalently, there exist integers $k$ and $l$ such that $Q(x_1, \dots, x_n) \not\equiv l \pmod{k}$ for all integer inputs. This phenomenon is known as a \emph{local obstruction}.

A natural question is when local obstructions arise and at which moduli. The subject has been studied by many methods; a recent contribution is the work of Liu and Ouyang~\cite{binary-modn}. However, the \emph{minimal} obstruction modulus of a given form has received little attention. We determine it completely for primitive binary quadratic forms and diagonal ternary quadratic forms.

\newpage
We first define $\kappa_Q$ and $\epsilon_{Q,p}$ for general quadratic forms.

\begin{dfn}
Let $Q$ be an integral positive definite quadratic form in $r$ variables.
\begin{enumerate}
\item For $l \in \Z$ and $k \in \Z_{\geq 1}$, we say that $Q \not\equiv l \pmod k$ if for all $x \in \Z^r$, we have $Q(x) \not\equiv l \pmod k$.
If $Q \not\equiv l \pmod k$, we say that \emph{$Q$ has a local obstruction modulo $k$}.
\item We denote the smallest modulus for which $Q$ has a local obstruction by
\begin{align*}
   \kappa_Q \coloneq \min \{k \in \Z_{\geq 1} \mid Q \text{ has a local obstruction modulo } k \}.
\end{align*} 
If $Q$ has no local obstruction, we set $\kappa_Q=\infty$.
\item For a prime $p$, we denote the smallest exponent for which $Q$ has a local obstruction by
\begin{align*}
   \epsilon_{Q,p} \coloneq \min \{e \in \Z_{\geq 1} \mid Q \text{ has a local obstruction modulo } {p^e} \},
\end{align*} 
where we set $ \epsilon_{Q,p} =\infty$ if $Q$ has no local obstruction modulo $p^e$ for any $e$.

\end{enumerate}
\end{dfn}
For example, if $\epsilon_{Q,2}=4$, $\epsilon_{Q,3}=2$, $\epsilon_{Q,5}=1$ and $\epsilon_{Q,p}= \infty$ for every prime $p \geq 7$,
then $\kappa_Q = \min \{ 2^4, 3^2, 5^1 \} = 5$.
By the Chinese remainder theorem, $Q$ represents every residue class modulo
$p_1^{e_1}\cdots p_s^{e_s}$ if and only if it represents every residue class modulo each
$p_i^{e_i}$. Thus, $Q$ has an obstruction modulo the product if and only if it has an obstruction modulo at least one prime-power factor. In particular,
\[
  \kappa_Q=\min_{\substack{p\text{ prime}\\ \epsilon_{Q,p}<\infty}}p^{\epsilon_{Q,p}}.
\]
As above, the minimum is understood to be $\infty$ when the indexing set is empty.

The main result of this note is an explicit formula for $ \kappa_Q $ and $\epsilon_{Q,p}$ for primitive positive definite integral binary quadratic forms.

\begin{thm}[Main Theorem]\label{thm:local-kappa}
Let $Q(x,y)=ax^2+bxy+cy^2$ be a primitive positive definite integral binary quadratic form with
discriminant $\Delta = b^2 - 4ac$.

\smallskip
\noindent\emph{For an odd prime $p$,} writing $\legendre{\Delta}{p}$ for the Legendre symbol,
\[
   \epsilon_{Q,p} =
   \begin{cases}
      1      & \text{if } \legendre{\Delta}{p} = 0 \quad (\text{i.e. } p \mid \Delta), \\[2pt]
      2      & \text{if } \legendre{\Delta}{p} = -1, \\[2pt]
      \infty & \text{if } \legendre{\Delta}{p} = 1.
   \end{cases}
\]

\noindent\emph{For $p = 2$:}
\[
   \epsilon_{Q,2} =
   \begin{cases}
      \infty & \text{if } \Delta \equiv 1 \pmod{8}, \\[2pt]
      2      & \text{if } \Delta \equiv 5 \pmod{8}, \\[2pt]
      3      & \text{if } 4 \mid \Delta \text{ and } \Delta = -2^{3}m \ (m \text{ positive and odd}), \\[2pt]
      2      & \begin{gathered}
                 \text{if }4 \mid \Delta,\ \Delta = -2^{n}m,\text{ and }m\text{ is positive and odd},\\
                 n\geq2\text{ and }n \neq 3
               \end{gathered}.
   \end{cases}
\]
\end{thm}

Rewriting this theorem in terms of $\kappa_Q$, we obtain the following.

\begin{col}[Main Theorem of $\kappa_Q$]\label{main-thm1}
Let $Q(x,y) = ax^2+bxy+cy^2$ be a primitive positive definite integral binary quadratic form and $\Delta = b^2 -4ac$.
Then $\kappa_Q$ depends only on $\Delta$ and is given as follows.

If $4 \mid \Delta$, write $\Delta=-2^n m$ with $n\geq2$ and $m$ positive and odd. When $m>1$, let
$p_{\min}$ be the smallest odd prime divisor of $\Delta$. Then
\[
\kappa_Q =
\begin{cases}
8 & \text{if }m=1\text{ and }n=3, \\
4 & \text{if }m=1\text{ and }n \neq 3, \\
\min \{ 8, p_{\min} \} & \text{if }m \neq 1\text{ and }n=3, \\
\min \{ 4, p_{\min} \} & \text{if }m \neq 1\text{ and }n \neq 3.
\end{cases}
\]

\medskip
\noindent Suppose that $\Delta \equiv 1 \pmod 4$, and set
\[
p_{\min}\coloneq \min \{ p \mid p \text{ is an odd prime divisor of } \Delta \},
\qquad
q_{\min} \coloneq \min\left\{ q \text{ odd prime} \mathrel{\Big|} \legendre{\Delta}{q} = -1 \right\}.
\]
Then
\[
\kappa_Q=
\begin{cases}
3 & \text{if }3\mid\Delta,\\
4 & \text{if }3\nmid\Delta\text{ and }\Delta\equiv5\pmod8,\\
\min\{p_{\min},q_{\min}^2\}
  & \text{if }3\nmid\Delta\text{ and }\Delta\equiv1\pmod8.
\end{cases}
\]
Here $\legendre{\cdot}{\cdot}$ denotes the Legendre symbol.
\end{col}

\begin{rem}
We have seen that in the case of binary quadratic forms, $\kappa_Q$ depends only on the discriminant $\Delta$. In general, this is not the case. For example, consider the following ternary quadratic forms and their minimal obstruction moduli:
\begin{enumerate}
\item[(1)] $Q_1(x,y,z) = x^2 + y^2 + 8z^2$, \quad$\kappa_{Q_1} = 4$.
\item[(2)] $Q_2(x,y,z) = x^2 + 2y^2 + 4z^2$, \quad$\kappa_{Q_2} = 16$.
\end{enumerate}
Both forms have discriminant $-64$ (using the signed determinant of the Hessian), but their minimal obstruction moduli differ.
\end{rem}

\begin{thm}[Main theorem for ternary diagonal forms]
Let $Q(x,y,z)= ax^2+by^2+cz^2$, where $a,b,c$ are positive integers and $\gcd(a,b,c)=1$.
Here we adopt the following convention: for an odd prime $p$ dividing exactly one of the coefficients, we label the coefficients so that $p \mid a$; for $p=2$, if exactly two coefficients are even we label them $a,b$, ordered so that $a \equiv 2 \pmod 4$ whenever exactly one of them is $\equiv 2 \pmod 4$; if exactly one coefficient is even we call it $a$. Then the case distinctions below are exhaustive.

\smallskip
\noindent\emph{For an odd prime $p$:}
\[
   \epsilon_{Q,p} =
   \begin{cases}
      1      & \text{if } p \text{ divides exactly two of the integers } a,b,c, \\[2pt]
      2      & \text{if } p \text{ divides exactly one of the integers } a,b,c \text{ and } \legendre{-bc}{p} = -1, \\[2pt]
      \infty & \text{if } p \nmid abc \text{, or } p \text{ divides exactly one of the integers } a,b,c \text{ and } \legendre{-bc}{p} = 1.
   \end{cases}
\]

\noindent\emph{For $p = 2$:}
\begin{enumerate}
\item[(1)] If all of $a, b, c$ are odd,
\[
   \epsilon_{Q,2} =
   \begin{cases}
      3      & \text{if } a \equiv b \equiv c \pmod 4, \\
      \infty      & \text{(otherwise)}.
   \end{cases}
\]

\item[(2)] If  $a, b$ are even and $c$ is odd,
\[
   \epsilon_{Q,2} =
   \begin{cases}
      2      & \text{if } 4 \mid a,b, \\
      3      & \text{if } a \equiv b \equiv 2 \pmod 4 \text{, or } a \equiv 2 \pmod 4 \text{ and } b \equiv 0 \pmod 8, \\
      4      & \text{if } a \equiv 2 \pmod 4 \text{ and } b \equiv 4 \pmod 8.
   \end{cases}
\]

\item[(3)] If $a$ is even and $b,c$ are odd,
\[
   \epsilon_{Q,2} =
   \begin{cases}
      2      & \text{if } 4 \mid a, \\
      \infty      & \text{if } a \equiv 2 \pmod 4 \text{ and } a+b+c \equiv 0 \pmod 8 \text{, or }\\
                  & a \equiv 2 \pmod 4 \text{ and } b+c \equiv 0 \pmod 8, \\
      4      &\text{(otherwise)}.
   \end{cases}
\]
\end{enumerate}
\end{thm}

\begin{col}[Main theorem for $\kappa_Q$, ternary case]\label{main-thm2}
Let $Q(x,y,z)= ax^2+by^2+cz^2$, where $a,b,c$ are positive integers and $\gcd(a,b,c)=1$.
Define
\[
p_{\min}\coloneq \min \{ p \text{ odd prime} \mid p \text{ divides exactly two of }a,b,c \}.
\]
For each odd prime $q$ dividing exactly one coefficient, let $r_q$ and $s_q$ denote the other two coefficients, and define
\[
q_{\min}\coloneq \min\left\{q \text{ odd prime} \mathrel{\Big|}
q \text{ divides exactly one of }a,b,c\text{ and }\legendre{-r_q s_q}{q}=-1\right\}.
\]
We use the convention $\min\varnothing=\infty$.
For the $2$-adic cases below, we use the coefficient-labeling convention of the preceding theorem.

\begin{enumerate}
\item[(1)] If all of $a, b, c$ are odd,
\[
   \kappa_Q =
   \begin{cases}
      \min \{ 8,p_{\min} \}      & \text{if } a \equiv b \equiv c \pmod 4, \\
      \min \{ p_{\min}, q_{\min}^2 \}      & \text{(otherwise)}.
   \end{cases}
\]

\item[(2)] If exactly two of $a, b, c$ are even,
\[
   \kappa_Q =
   \begin{cases}
      \min \{ 4, p_{\min} \}      & \text{if } 4 \mid a,b, \\
      \min \{ 8, p_{\min} \}      & \text{if } a \equiv b \equiv 2 \pmod 4 \text{, or } a \equiv 2 \pmod 4 \text{ and } b \equiv 0 \pmod 8, \\
      \min \{ 16, p_{\min}, q_{\min}^2 \}      & \text{if } a \equiv 2 \pmod 4 \text{ and } b \equiv 4 \pmod 8.
   \end{cases}
\]

\item[(3)] If exactly one of $a, b, c$ is even,
\[
   \kappa_Q =
   \begin{cases}
      \min \{ 4, p_{\min} \}      & \text{if } 4 \mid a, \\
      \min \{ p_{\min}, q_{\min}^2 \}      & \text{if } a \equiv 2 \pmod 4 \text{ and } a+b+c \equiv 0 \pmod 8 \text{, or }\\
                  & a \equiv 2 \pmod 4 \text{ and } b+c \equiv 0 \pmod 8, \\
      \min \{ 16, p_{\min}, q_{\min}^2 \}      &\text{(otherwise)}. 
   \end{cases}
\]
\end{enumerate}
\end{col}

When $Q(x,y,z)$ is not diagonal, the presence of numerous cross terms renders techniques
such as completing the square difficult to apply. Thus, explicitly
determining $\epsilon_{Q,p}$ and $\epsilon_{Q,2}$ in particular appears
to be considerably more difficult.

\section{Binary case}
In this section, we assume that $Q$ is a primitive positive definite integral binary quadratic form.

\subsection{The case \texorpdfstring{$\Delta \equiv 0 \pmod {4}$}{Delta congruent to 0 modulo 4}}
In this case, we will need the following lemma, which is a classical result of the theory of 
quadratic forms. For example, it is mentioned in a more general form in~\cite{cassels}. We will give an elementary proof for the special case of quadratic
forms. 

\begin{lem}\label{lem1}
Let $p$ be an odd prime.
\begin{enumerate}
\item[(1)] If $4 \mid \Delta$, then $Q(x,y)$ has an obstruction modulo $8$.
\item[(2)] $Q(x,y)$ has no obstruction modulo $2$.
\item[(3)] We have $\gcd(p,\Delta) = 1$ if and only if $Q(x,y)$ has no obstruction modulo $p$.
\end{enumerate}
\end{lem}
\begin{proof} Let $Q(x,y) = ax^2+bxy+cy^2$ be a binary quadratic form.
\begin{enumerate}
\item[(1)] Since $4 \mid \Delta$, the coefficient $b$ is even, say $b = 2h$; since $Q(x,y)$ is primitive, at least one of $a$ and $c$ is odd, and interchanging $a$ and $c$ if necessary we may assume that $a$ is odd.
Completing the square, we get $Q(x,y) = ax^2+2hxy+c y^2 = a(x + \frac{h}{a} y)^2 + (c- \frac{h^2}{a}) y^2$.
Since $a^2 \equiv 1 \pmod 8$, we have $a^{-1} \equiv a \pmod 8$ and hence $Q \equiv a(x + ha y)^2 + (c- h^2 a) y^2 \pmod 8$.
Since the map
\begin{align*}
f:(\Z /8\Z)^2 &\longrightarrow (\Z /8\Z)^2,\\
(x,y) &\longmapsto (x + ha y , y) =: (X,Y),
\end{align*}
is bijective, the value sets of $Q$ and of $aX^2 + c' Y^2$ modulo $8$ coincide, where $c' = c-h^2 a$.
Thus, using $X^2 \equiv 0,1,4 \pmod 8$, the value set of $Q$ modulo $8$ is
\[
\mathcal V_8(Q) = \{0,a,4a,c',a+c',4a+c',4c',a+4c',4(a+c')\}.
\]
If $c'$ is even, we have $4c' \equiv 0 \pmod 8$ and therefore
\[
\#\mathcal V_8(Q) = \#\{0,a,4a,c',a+c',4a+c'\} \leq 6.
\]
If $c'$ is odd, then $4c' \equiv 4a \pmod 8$ and $4(a+c') \equiv 0 \pmod 8$, and therefore
\[
\#\mathcal V_8(Q) = \#\{0,a,4a,c',a+c',4a+c',a+4c'\}\leq 7.
\]
In both cases $\#\mathcal V_8(Q)<8$, and therefore $Q(x,y)$ has an obstruction modulo $8$.

\item[(2)] We write $[a,b,c]$ for the form $ax^2+bxy + cy^2$ with coefficients reduced modulo $2$.
Since $Q(x,y)$ is primitive, $Q$ is congruent modulo $2$ to one of \\
$\{[0,0,1],[0,1,0],[0,1,1],[1,0,0],[1,0,1],[1,1,0],[1,1,1]\}$.
None of these forms has an obstruction modulo $2$.

\item[(3)($\Rightarrow$)] After possibly interchanging $a$ and $c$, we may assume that either $\gcd(a,p) = 1$, or $p \mid a,c$ and $p \nmid b$.

\item[(3-i)]$\gcd(a,p) = 1$ \\
Since $\gcd(4a,p)=1$, multiplication by $4a$ permutes the residue classes modulo $p$, so we may consider $4aQ$ instead of $Q(x,y)$.
Completing the square, we have $4aQ \equiv (2ax+by)^2 - \Delta y^2 \pmod p$.
Similarly to (1), it suffices to consider $X^2 - \Delta Y^2 \pmod p$.
Fix $k \in \Z /p\Z$ and consider the two sets $ \{X^2 \pmod p\}$ and $\{\Delta Y^2 + k \pmod p \}$;
we have $\# \{X^2 \pmod p\}$ = $\# \{\Delta Y^2 + k \pmod p \} = \frac{p+1}{2}$.
By the pigeonhole principle, the two sets intersect, i.e.\ there exist $X_0,Y_0$ such that
$X_0^2 - \Delta Y_0^2 \equiv k \pmod p$.
Thus, $4aQ$, and hence $Q(x,y)$, has no obstruction modulo $p$.

\item[(3-ii)]$p \mid a,c$ and $p \nmid b$ \\
By hypothesis, we get $Q \equiv bxy \pmod p$.
Taking $x \equiv b^{-1}$ and $y \equiv k \pmod p$, we have $Q \equiv k \pmod p$.
Hence, $Q(x,y)$ has no obstruction modulo $p$.

\item[($\Leftarrow$)] We prove the contrapositive: if $p \mid \Delta$, then $Q(x,y)$ has an obstruction modulo $p$.
If $p \nmid a$, the calculation in (3-i) gives $4aQ \equiv (2ax+by)^2 \pmod p$. Thus, $Q$ takes at most $\frac{p+1}{2} < p$ values modulo $p$.
If $p \mid a$, then $p \mid b^2 = \Delta + 4ac$, so $p \mid b$ and $Q \equiv cy^2 \pmod p$ with $p \nmid c$ by primitivity, again at most $\frac{p+1}{2}$ values.
Therefore, $Q(x,y)$ has an obstruction modulo $p$. This proves the lemma.

\end{enumerate}
\end{proof}

We now give the proof of \cref{main-thm1} in the case when $4 \mid \Delta$.

\begin{prop}\label{prop:min-obs-even}
Let $4 \mid \Delta$ and write $\Delta = -2^n m$ with $n \geq 2$ and $m$ a positive odd integer.
Let $p_{\min}\coloneq \min \{ p \mid p \text{ is an odd prime divisor of } \Delta \}$.
We set $p_{\min}=\infty$ when $\Delta$ has no odd prime divisor.
Then
\[
\kappa_Q = 
\begin{cases}
8 & \text{if } m=1 \text{ and } n=3, \\
4 & \text{if } m=1 \text{ and } n \neq 3, \\
\min \{ 8, p_{\min} \} & \text{if } m \neq 1 \text{ and } n=3, \\
\min \{ 4, p_{\min} \} & \text{if } m \neq 1 \text{ and } n \neq 3.
\end{cases}
\]
\end{prop}

\begin{proof}
Since $Q(x,y)$ is primitive and $4 \mid \Delta$, at least one of $a$ and $c$ is odd. After interchanging $a$ and $c$ if necessary, we may assume that $a$ is odd. Also, $b = 2h$ for some integer $h$.
Note that $\Delta = 4(h^2 - ac)$, so $h^2 - ac = -2^{n-2}m$.
By \cref{lem1}, $Q(x,y)$ has no obstruction modulo $2$, has an obstruction modulo $8$,
and has no obstruction modulo any odd prime $p$ with $\gcd(p,\Delta)=1$.
Therefore, the value of $\kappa_Q$ is determined by whether $Q(x,y)$ has an obstruction modulo $4$:
if it does, then $\kappa_Q = \min\{4,\, p_{\min}\}$ (with $p_{\min} = \infty$ when $m=1$);
if it does not, then $\kappa_Q = \min\{8,\, p_{\min}\}$.

It thus suffices to prove the following claim.

\begin{clm}
$Q(x,y)$ has no obstruction modulo $4$ if and only if $n = 3$.
\end{clm}

\medskip
\noindent\textbf{Case $n = 3$} (i.e.\ $h^2 - ac \equiv 2 \pmod{4}$)\textbf{.}\;
We show that $Q(x,y)$ has no obstruction modulo $4$.

\begin{enumerate}
  \item If $h$ is even, then $ac \equiv 2 \pmod{4}$, so
  $(a \bmod 4,\, c \bmod 4) \in \{(1,2),\,(3,2)\}$.
  Hence $Q \equiv \pm x^2 + 2y^2 \pmod{4}$, and a direct check shows that
  every residue class modulo $4$ is represented.

  \item If $h$ is odd, then $ac \equiv 3 \pmod{4}$, so
  $(a \bmod 4,\, c \bmod 4) \in \{(1,3),\,(3,1)\}$.
  Hence $Q \equiv x^2 + 2xy - y^2$ or $Q \equiv -x^2 + 2xy + y^2 \pmod{4}$
  respectively, and again every residue class modulo $4$ is represented.
\end{enumerate}

\medskip
\noindent\textbf{Case $n \neq 3$} (i.e.\ $h^2 - ac \equiv 0, 1, \text{ or } 3 \pmod{4}$)\textbf{.}\;
We show that $Q(x,y)$ has an obstruction modulo $4$ by examining each subcase.

\begin{enumerate}
  \item If $h^2 - ac \equiv 0 \pmod{4}$ and $h$ is even, then
  $ac \equiv 0 \pmod{4}$, so $c \equiv 0 \pmod{4}$ (since $a$ is odd),
  giving $Q \equiv ax^2 \pmod{4}$.

  \item If $h^2 - ac \equiv 0 \pmod 4$ and $h$ is odd, then
  $ac \equiv 1 \pmod 4$, so
  $(a \bmod 4,\, c \bmod 4) \in \{(1,1),\,(3,3)\}$,
  giving $Q \equiv \pm(x^2 + 2xy + y^2) \pmod{4}$.

  \item If $h^2 - ac \equiv 1 \text{ or } 3 \pmod{4}$ and $h$ is even, \\ 
  then $(a \bmod 4,\, c \bmod 4) \in \{(1,1),\,(3,3),\,(1,3),\,(3,1)\}$,
  giving $Q \equiv \pm x^2 \pm y^2 \pmod{4}$.

  \item If $h^2 - ac \equiv 1 \text{ or } 3 \pmod{4}$ and $h$ is odd, \\ 
  then $(a \bmod 4,\, c \bmod 4) \in \{(1,0),\,(3,0),\,(1,2),\,(3,2)\}$,
  giving $Q \equiv \pm x^2 + 2xy \pmod{4}$ or
  $Q \equiv \pm x^2 + 2xy + 2y^2 \pmod{4}$.
\end{enumerate}

In each subcase, a direct check confirms that $Q(x,y)$ has an obstruction
modulo $4$.
\end{proof}

\subsection{The case \texorpdfstring{$\Delta \equiv 1 \pmod 4$}{Delta congruent to 1 modulo 4}}

Note that $b$ is odd in this case.
The argument splits further into three subcases. First, we consider the case $\Delta \equiv 0 \pmod 3$.
\begin{prop}\label{prop:kappa-three}
Let $Q(x,y)$ be a primitive positive definite integral binary quadratic form with $\Delta \equiv 1 \pmod 4$
and $\Delta \equiv 0 \pmod 3$. Then
$\kappa_Q = 3$.
\end{prop}

\begin{proof}
By \cref{lem1}\,(2), $Q(x,y)$ has no obstruction modulo~$2$,
whereas $Q(x,y)$ has an obstruction modulo~$3$ since $3 \mid \Delta$.
Therefore, $\kappa_Q = 3$.
\end{proof}
Next, we consider the case $3 \nmid \Delta$ and $\Delta \equiv 5 \pmod 8$.
\begin{prop}
Let $Q(x,y)$ be a primitive positive definite integral binary quadratic form whose discriminant satisfies
$3 \nmid \Delta$ and $\Delta \equiv 5 \pmod 8$. Then
$\kappa_Q = 4$; moreover, $Q \not\equiv 2 \pmod 4$.
\end{prop}

\begin{proof}
Since $3 \nmid \Delta$ and $Q$ is primitive, $Q(x,y)$ has no obstruction modulo $2$ or $3$ by \cref{lem1}(2) and (3).
So we consider $Q(x,y) \pmod 4$.
Since $b$ is odd and $b^2 - 4ac \equiv 5 \pmod{8}$,
it follows that $a$, $b$, and $c$ are all odd.
Hence $Q(x,y)$ is equivalent modulo~$4$ to one of
\[
    [1,1,1],\quad [1,1,3],\quad [1,3,1],\quad [1,3,3],\quad
    [3,1,1],\quad [3,1,3],\quad [3,3,1],\quad [3,3,3].
\]
Since each of these eight forms has an obstruction modulo $4$
and none of them represents $2$ modulo $4$, the proposition follows.
\end{proof}

Finally, we treat the case $3 \nmid \Delta$ and $\Delta \equiv 1 \pmod 8$.

\begin{prop}\label{prop:min-obs-odd}
Let $Q(x,y)$ be a primitive positive definite integral binary quadratic form whose discriminant satisfies
$3 \nmid \Delta$ and $\Delta \equiv 1 \pmod 8$.
Let $p_{\min}$ be the smallest odd prime divisor of $\Delta$ and let
$q_{\min} \coloneq \min\{ q \text{ odd prime} \mid \legendre{\Delta}{q} = -1 \}$, where $\legendre{\cdot}{\cdot}$ is the Legendre symbol. Then
\[
\kappa_Q = \min \{ p_{\min} , q_{\min}^2 \}.
\]
\end{prop}

This proposition follows from the next lemma.
We call a solution $(x,y)$ of a congruence \emph{non-trivial modulo $m$} if $(x,y) \not\equiv (0,0) \pmod m$.

\begin{lem}\label{lem2}
Let $q$ be an odd prime and let $k$ be a positive integer.
\begin{enumerate}
\item[(1)] If $\Delta \equiv 1 \pmod 8$, then each of the congruences $Q(x,y) \equiv 0 \pmod 2$ and $Q(x,y) \equiv 1 \pmod 2$ has a non-trivial solution modulo $2$.
\item[(2)] If $\Delta \equiv 1 \pmod 8$, then $Q(x,y)$ has no obstruction modulo $2^k$.
\item[(3)] If $\legendre{\Delta}{q} = -1$, then the congruence $Q(x,y) \equiv 0 \pmod q$ has no non-trivial solution modulo $q$.
Consequently, $Q(x,y)$ has an obstruction modulo $q^2$; for example, $Q \not\equiv q \pmod {q^2}$.
\item[(4)] If $\legendre{\Delta}{q} = 1$, then $Q(x,y)$ has no obstruction modulo $q^k$.
\end{enumerate}
\end{lem}

These results are special cases of classical results in the theory of quadratic forms (see~\cite{coxprime,cassels}). We provide an elementary proof
for the case of binary quadratic forms.

\begin{proof}[Proof of \cref{lem2}]
Throughout the proof, $u,v$ denote integers.

\begin{enumerate}
    \item[(1)] Since $b$ is odd and $b^2 - 4ac \equiv 1 \pmod{8}$, at least one of $a$ and $c$ is even.
    Hence $Q(x,y)$ is equivalent modulo $2$ to one of $[1,1,0]$, $[0,1,1]$, or $[0,1,0]$.
    Taking $(x,y) = (0,1)$, $(1,0)$, or $(0,1)$ respectively, we obtain a non-trivial solution of
    $Q \equiv 0 \pmod{2}$. 
    If $Q \equiv 1 \pmod 2$, taking $(x,y) = (1,0)$, $(0,1)$, or $(1,1)$ respectively, we obtain a non-trivial solution of
    $Q \equiv 1 \pmod{2}$.

    \item[(2)] We prove by induction that every residue class modulo $2^k$ has a non-trivial representation. The base case is (1). Let $N$ be a residue class modulo $2^{k+1}$, and choose a non-trivial solution $(x_k,y_k)$ of
    $Q(x_k,y_k)\equiv N\pmod{2^k}$. Set
    $(x_{k+1}, y_{k+1}) \coloneqq (x_k + 2^k u,\, y_k + 2^k v)$.
    A direct computation gives
    \[
        Q(x_{k+1}, y_{k+1}) \equiv Q(x_k,y_k) + 2^k(bx_k + 2cy_k)v + 2^k(2ax_k + by_k)u \pmod{2^{k+1}}.
    \]
    Since $b$ is odd and $(x_k, y_k)$ is non-trivial modulo $2$, one can choose $u, v \in \mathbb{Z}$ so that
    $bx_k v + by_k u$ attains either parity. This allows the last binary digit to be chosen so that $Q(x_{k+1},y_{k+1})\equiv N\pmod{2^{k+1}}$. Moreover, $(x_{k+1},y_{k+1}) \equiv (x_k,y_k) \pmod 2$ remains non-trivial, completing the induction.
    \item[(3)] We prove the contrapositive.
    Let $(x,y)$ be a non-trivial solution of $Q \equiv 0 \pmod{q}$.
    If necessary, applying the bijections $(x,y) \mapsto (y,x)$ or $(x,y) \mapsto (x, x+y)$,
    we may assume $q \nmid a$ and $y \not\equiv 0 \pmod{q}$.
    By an argument analogous to that of \cref{lem1}(3-i), the congruence $Q\equiv0\pmod q$ is equivalent to $X^2 - \Delta Y^2\equiv0\pmod q$.
    Since $Y \not\equiv 0 \pmod{q}$, we have $\left(\tfrac{X}{Y}\right)^2 \equiv \Delta \pmod{q}$,
    which implies $\legendre{\Delta}{q} \in \{0, 1\}$. This establishes the contrapositive.
    Under the hypothesis $\legendre{\Delta}{q}=-1$, it follows that $Q(x,y)\equiv0\pmod q$ only when $q\mid x,y$, and hence $Q(x,y)\equiv0\pmod{q^2}$.
    Thus, $Q \not\equiv q \pmod{q^2}$.

    \item[(4)] Since $\legendre{\Delta}{q} = 1$, Hensel's lemma shows that $\Delta$ is also a square modulo $q^k$;
    that is, there exists $z_k \in \mathbb{Z}/q^k\mathbb{Z}$ satisfying ${z_k}^2 \equiv \Delta \pmod{q^k}$.
    Note that $z_k \not\equiv 0 \pmod{q}$.
    Arranging $q \nmid a$ as in (3) and completing the square as in the proof of \cref{lem1}(3-i), we have
    \[
        4aQ \equiv X^2 - \Delta Y^2 \equiv X^2 - {z_k}^2 Y^2 \equiv (X + z_k Y)(X - z_k Y) \pmod{q^k}.
    \]
    Setting
    \[
        X \equiv \frac{1 + m}{2} \pmod{q^k}
        \quad \text{and} \quad
        Y \equiv \frac{1 - m}{2z_k} \pmod{q^k},
    \]
    where $m$ is any residue class modulo $q^k$, we obtain $(X + z_k Y)(X - z_k Y) \equiv m \pmod{q^k}$.
    Thus, $4aQ$ represents every residue class modulo $q^k$. Since multiplication by $4a$ permutes the residue classes modulo $q^k$, it follows that $Q$ has no obstruction modulo $q^k$.
\end{enumerate}
\end{proof}

\begin{proof}[Proof of \cref{prop:min-obs-odd}]
By the Chinese remainder theorem and \cref{lem2}(2), $\kappa_Q$ is a power of an odd prime $q$.
Moreover, by \cref{lem2}(4), this $q$ satisfies $\legendre{\Delta}{q} \in \{0,-1\}$.
If $\legendre{\Delta}{q} = 0$, that is, if $q \mid \Delta$, then $Q$ has an obstruction already modulo $q$, and the smallest such modulus is $p_{\min}$.
If $\legendre{\Delta}{q} =-1$, then \cref{lem1}(3) and \cref{lem2}(3) show that the first obstruction occurs modulo $q^2$; the smallest such modulus is $q_{\min}^2$.
Therefore, $\kappa_Q = \min \{ p_{\min} , q_{\min}^2 \}$.
\end{proof}


\begin{ex} Using \cref{main-thm1}, we obtain the following examples.
\begin{enumerate}
\item[(1)] $Q=x^2+2y^2, \Delta = -8 \Rightarrow \kappa_Q=8, Q \not\equiv 5,7 \pmod 8$.
\item[(2)] $Q=x^2+4y^2, \Delta = -16 \Rightarrow \kappa_Q=4, Q \not\equiv 2,3 \pmod 4$.
\item[(3)] $Q=x^2+2xy+4y^2, \Delta = -12 \Rightarrow \kappa_Q=3, Q \not\equiv 2 \pmod 3$.
\item[(4)] $Q=2x^2+4xy+7y^2, \Delta = -40 \Rightarrow \kappa_Q=5, Q \not\equiv 1,4 \pmod 5$.
\item[(5)] $Q=x^2+xy+10y^2, \Delta = -39 \Rightarrow \kappa_Q=3, Q \not\equiv 2 \pmod 3$.
\item[(6)] $Q=x^2+3xy+13y^2, \Delta = -43 \Rightarrow \kappa_Q=4, Q \not\equiv 2 \pmod 4$.
\item[(7)] $Q=x^2+3xy+14y^2, \Delta = -47, p_{\min}=47, q_{\min}=5  \Rightarrow \kappa_Q=25, Q \not\equiv 5,10,15,20 \pmod {25}$.
\end{enumerate}
\end{ex}

\section{Local obstruction exponents for ternary diagonal forms}
In this section, let $Q(x,y,z) = ax^2+by^2+cz^2$ be a primitive positive definite integral ternary diagonal form.
Using the results above, we can completely describe $\epsilon_{Q,p}$ for such forms when $p$ is odd.
Since $Q$ is diagonal, by permuting $x, y, z$ if necessary,
we may assume that $p \mid a$ when $p$ divides exactly one of $a, b, c$, and that $p \mid a, b$ when $p$ divides exactly two of $a, b, c$.
We henceforth assume that $Q$ is of this shape.

\begin{prop}[$\epsilon_{Q,p}$ for odd primes]\label{prop:ternary-odd-primes} Let $Q(x,y,z) = ax^2+by^2+cz^2$ be a primitive positive definite integral ternary diagonal form. Then for odd primes
$p$, we have
\[
   \epsilon_{Q,p} =
   \begin{cases}
      1      & \text{if } p \text{ divides exactly two of the integers } a,b,c, \\[2pt]
      2      & \text{if } p \text{ divides exactly one of the integers } a,b,c \text{ and } \legendre{-bc}{p} = -1, \\[2pt]
      \infty & \text{if } p \nmid abc \text{, or } p \text{ divides exactly one of the integers } a,b,c \text{ and } \legendre{-bc}{p} = 1.
   \end{cases}
\]
\end{prop}

\begin{lem}\label{lem:p-adic-lift}
Let $Q(x,y,z) = ax^2 + by^2 + cz^2$, let $p$ be an odd prime, and suppose $p \nmid abc$.
If $Q(x_1,y_1,z_1) \equiv n_1 \pmod{p}$ and $(x_1,y_1,z_1)$ is non-trivial modulo $p$,
then for every $k \geq 2$, every integer $n$ with $n \equiv n_1 \pmod p$ is represented by $Q$ modulo $p^k$.
\end{lem}

The following lifting arguments are similar to the multivariable polynomial version of Hensel's lemma~\cite{hensel-multivariable}. Because the two $2$-adic arguments below begin at different powers of $2$, we give direct proofs.

\begin{proof}[Proof of \cref{lem:p-adic-lift}]
Fix an integer $n\equiv n_1\pmod p$. We proceed by induction on $k$.

Suppose $Q(x_{k-1},y_{k-1},z_{k-1}) \equiv n \pmod{p^{k-1}}$ and that $(x_{k-1},y_{k-1},z_{k-1})$ is
non-trivial modulo $p$.

Define
\[
x_k = x_{k-1} + p^{k-1}u, \qquad y_k = y_{k-1} + p^{k-1}v, \qquad z_k = z_{k-1} + p^{k-1}w.
\]
Then
\[
Q(x_k,y_k,z_k) \equiv Q(x_{k-1},y_{k-1},z_{k-1}) + p^{k-1}\bigl(2ax_{k-1}u + 2by_{k-1}v + 2cz_{k-1}w\bigr) + p^{2k-2}(\dots) \pmod{p^k}.
\]
Since $2k-2 \geq k$, this simplifies to
\[
Q(x_k,y_k,z_k) \equiv Q(x_{k-1},y_{k-1},z_{k-1}) + p^{k-1}\bigl(2ax_{k-1}u + 2by_{k-1}v + 2cz_{k-1}w\bigr) \pmod{p^k}.
\]
Because $p\nmid abc$ and the solution is non-trivial modulo $p$, at least one coefficient of $u,v,w$ below is nonzero modulo $p$. We can therefore arrange
\[
2ax_{k-1}u + 2by_{k-1}v + 2cz_{k-1}w \equiv 0 ,1 , \dots , p-1 \pmod{p},
\]
and hence choose the next base-$p$ digit so that $Q(x_k,y_k,z_k)\equiv n\pmod{p^k}$. The new solution is congruent to the old one modulo $p$, so it remains non-trivial. This completes the induction.
\end{proof}

\begin{proof}[Proof of \cref{prop:ternary-odd-primes}]
\begin{enumerate}
\item[(1)]
Suppose $p \nmid abc$. It suffices to exhibit, for every residue $n$ modulo $p$, a non-trivial solution of $Q(x,y,z) \equiv n \pmod{p}$; then \cref{lem:p-adic-lift} shows that $Q$ has no obstruction modulo $p^k$ for any $k$.
If $n \not\equiv 0 \pmod p$, setting $z = 0$ reduces the congruence to $ax^2 + by^2 \equiv n \pmod{p}$,
which has a solution by the proof of \cref{lem1}(3); since $n \not\equiv 0$, this solution is non-trivial modulo $p$.
If $n \equiv 0 \pmod p$, setting $z \equiv 1 \pmod{p}$ reduces the congruence to
\[
  ax^2 + by^2 \equiv -c \pmod{p}.
\]
Again by the proof of \cref{lem1}(3), there exists a solution $(x_1, y_1)$, and
$(x_1, y_1, 1)$ is a non-trivial solution of $Q(x,y,z) \equiv 0 \pmod{p}$.

\item[(2)]
Suppose $p$ divides exactly two of the integers $a, b, c$, so that
$Q(x,y,z) \equiv cz^2 \pmod{p}$.
Since $\frac{p+1}{2} < p$, it follows that $Q(x,y,z)$ has an obstruction modulo $p$.

\item[(3)]
Suppose $p$ divides exactly one of the integers $a, b, c$, so that
$Q(x,y,z) \equiv by^2 + cz^2 \pmod{p}$.

\begin{enumerate}

\item[(3-i)]
If $\left(\frac{-bc}{p}\right) = 1$,
then by an argument analogous to the proof of \cref{lem2}(4),
$Q(x,y,z)$ has no obstruction modulo $p^k$.
  \item[(3-ii)]
  If $\left(\frac{-bc}{p}\right) = -1$, then, since $p \nmid bc$ and $Q(x,y,z) \equiv by^2 + cz^2 \pmod{p}$, the pigeonhole principle shows that $Q(x,y,z)$ has no obstruction modulo $p$.
  So we consider $Q(x,y,z) \pmod {p^2}$. By an argument analogous to the proof of \cref{lem2}(3),
  every solution of $Q \equiv 0 \pmod{p}$ must have $(y, z)$ trivial modulo $p$.
  Writing $a = pa'$, $y = py'$, and $z = pz'$, we obtain
  \[
    Q(x,y,z) = pa'x^2 + bp^2y'^2 + cp^2z'^2 \equiv pa'x^2 \pmod{p^2}.
  \]
  In either case, whether $p \mid a'$ or not, since $\frac{p+1}{2} < p$, we have
  $\#\!\left\{\, pa'x^2 \pmod{p^2} \,\right\} < p$,
  so some residue class divisible by $p$ is not represented modulo $p^2$,
  and therefore $Q(x,y,z)$ has an obstruction modulo $p^2$.
  This completes the proof of the proposition.
\end{enumerate}
\end{enumerate}
\end{proof}

For $\epsilon_{Q,2}$, we prove the claim by dividing into three cases according to the number of even coefficients:
(i) all of $a, b, c$ are odd; (ii) exactly two of $a, b, c$ are even; and (iii) exactly one of $a, b, c$ is even.
Note that $\epsilon_{Q,2} \geq 2$ because $Q(x,y,z)$ is primitive.
\subsection{(i) All \texorpdfstring{$a,b,c$}{a, b, c} are odd}

\begin{prop}\label{prop:ternary-all-odd}
Let $Q(x,y,z)=ax^2+by^2+cz^2$ with $a,b,c$ odd. Then
\[
   \epsilon_{Q,2} =
   \begin{cases}
      3      & \text{if } a \equiv b \equiv c \pmod 4, \\
      \infty      & \text{(otherwise)}.
   \end{cases}
\]
\end{prop}

\begin{lem}\label{lem:8-adic-lift}
Let $Q(x,y,z) = ax^2 + by^2 + cz^2$ be an integral diagonal form.
Suppose that $Q(x_3,y_3,z_3) \equiv n_3 \pmod{8}$ and at least one of
$ax_3$, $by_3$, and $cz_3$ is odd. Then, for every $k \geq 4$, every integer $n$ with $n \equiv n_3 \pmod 8$ is represented by $Q$ modulo $2^k$.
\end{lem}

\begin{proof}[Proof of \cref{lem:8-adic-lift}]
Fix an integer $n\equiv n_3\pmod8$. We proceed by induction on $k$.

Suppose $Q(x_{k-1},y_{k-1},z_{k-1}) \equiv n \pmod{2^{k-1}}$ and at least one of
$ax_{k-1}$, $by_{k-1}$, and $cz_{k-1}$ is odd.

Define
\[
x_k = x_{k-1} + 2^{k-2}u, \qquad y_k = y_{k-1} + 2^{k-2}v, \qquad z_k = z_{k-1} + 2^{k-2}w.
\]
Then
\[
Q(x_k,y_k,z_k) \equiv Q(x_{k-1},y_{k-1},z_{k-1}) + 2^{k-1}\bigl(ax_{k-1}u + by_{k-1}v + cz_{k-1}w\bigr) + 2^{2k-4}(\dots) \pmod{2^k}.
\]
Since $2k-4 \geq k$, this simplifies to
\[
Q(x_k,y_k,z_k) \equiv Q(x_{k-1},y_{k-1},z_{k-1}) + 2^{k-1}\bigl(ax_{k-1}u + by_{k-1}v + cz_{k-1}w\bigr) \pmod{2^k}.
\]
By choosing $u,v,w$ appropriately, we can arrange
\[
ax_{k-1}u + by_{k-1}v + cz_{k-1}w \equiv 0 \text{ or } 1 \pmod{2},
\]
and hence choose the next binary digit so that $Q(x_k,y_k,z_k)\equiv n\pmod{2^k}$. The parity of each coordinate is unchanged, so the hypothesis persists. This completes the induction.
\end{proof}

\begin{proof}[Proof of \cref{prop:ternary-all-odd}]
\begin{enumerate}
\item[(1)] Suppose $a \equiv b \equiv c \pmod{4}$.
Since $Q \equiv a(x^2+y^2+z^2) \pmod 4$ and $x^2+y^2+z^2$ attains every residue modulo $4$, the form $Q(x,y,z)$ has no obstruction modulo $4$.
We now determine the odd residues represented modulo $8$.
An odd value of $Q$ arises only when exactly one or all three of $x,y,z$ are odd; using $t^2 \equiv 1 \pmod 8$ for odd $t$ and $t^2 \equiv 0, 4 \pmod 8$ for even $t$, the odd values of $Q$ modulo $8$ are
\[
\{\, a,\ a+4,\ b,\ b+4,\ c,\ c+4,\ a+b+c \,\}.
\]
Since $a \equiv b \equiv c \pmod 4$, we have $\{a,a+4\} = \{b,b+4\} = \{c,c+4\}$, while $b + c \equiv 2a \equiv 2 \pmod 4$ gives $a+b+c \equiv a + 2 \pmod 4$.
Hence the odd residues represented modulo $8$ are exactly the three classes $\{a,\ a+4,\ a+b+c\}$, so one odd residue class is missed, and $Q(x,y,z)$ has an obstruction modulo $8$. Therefore $\epsilon_{Q,2} = 3$.

\item[(2)] Otherwise, we may assume $a \not\equiv b \pmod{4}$.
By \cref{lem:8-adic-lift}, it suffices to show that $Q(x,y,z) \equiv n \pmod{8}$ has a non-trivial solution for every $n \pmod{8}$. We now exhibit such solutions
explicitly.

For
\[
(x,y,z) = (1,0,0),\ (1,0,2),\ (0,1,0),\ (0,1,2) \pmod{8},
\]
the corresponding values of $Q(x,y,z)$ modulo $8$ are, respectively,
\[
a,\quad a+4,\quad b,\quad b+4.
\]
Since $a,b$ are odd and $a \not\equiv b \pmod{4}$, we have
\[
\{a,\ a+4,\ b,\ b+4\} = \{1,3,5,7\} \pmod{8}.
\]

On the other hand, by hypothesis $a+b \equiv 0 \pmod{4}$, and at least one of $b+c$, $c+a$ is $\equiv 2 \pmod{4}$. By relabeling the variables if necessary, 
we may assume $b+c \equiv 2 \pmod{4}$.

For
\[
(x,y,z) = (1,1,0),\ (1,1,2),\ (0,1,1),\ (2,1,1) \pmod{8},
\]
the corresponding values of $Q(x,y,z)$ modulo $8$ are, respectively,
\[
a+b,\quad a+b+4,\quad b+c,\quad b+c+4.
\]
By the assumptions above,
\[
\{a+b,\ a+b+4,\ b+c,\ b+c+4\} = \{0,2,4,6\} \pmod{8}.
\]

Therefore, by \cref{lem:8-adic-lift}, $Q(x,y,z)$ has no obstruction modulo $2^k$ for any $k\geq1$.
\end{enumerate}
\end{proof}

\subsection{(ii) Exactly two of \texorpdfstring{$a,b,c$}{a, b, c} are even}
As for $\epsilon_{Q,p}$, we may assume that $a,b$ are even and $c$ is odd.
\begin{prop}\label{prop:ternary-two-even}
Let $Q(x,y,z)=ax^2+by^2+cz^2$ with $a,b$ even and $c$ odd. Then
\[
   \epsilon_{Q,2} =
   \begin{cases}
      2      & \text{if } 4 \mid a,b, \\
      3      & \text{if } a \equiv b \equiv 2 \pmod 4 \text{, or } a \equiv 2 \pmod 4 \text{ and } b \equiv 0 \pmod 8, \\
      4      & \text{if } a \equiv 2 \pmod 4 \text{ and } b \equiv 4 \pmod 8.
   \end{cases}
\]
\end{prop}

\begin{proof}[Proof of \cref{prop:ternary-two-even}]
\begin{enumerate}
\item[(1)] Suppose $4 \mid a$ and $4 \mid b$. Then $Q(x,y,z) \equiv cz^2 \pmod{4}$. Since $cz^2 \not\equiv 2 \pmod{4}$, this proves the claim.

\item[(2)] Suppose $a \equiv b \equiv 2 \pmod{4}$, or $a \equiv 2 \pmod{4}$ and $b \equiv 0 \pmod{8}$. Then
\[
Q(x,y,z) \equiv 2x^2 \pm z^2 \ (+\, by^2) \pmod{4}.
\]
A direct check shows that these forms represent every residue class modulo $4$.

On the other hand,
\[
ax^2 \equiv 0, a \pmod{8}, \qquad by^2 \equiv 0, b \pmod{8}, \qquad cz^2 \equiv 0, c, 4 \pmod{8}.
\]
Considering the representable odd residues modulo $8$, we obtain
\[
Q_{\mathrm{odd}(8)} = \{\, c,\ a+c,\ b+c,\ a+b+c \,\}.
\]
Hence, to complete the proof it suffices to show $\#Q_{\mathrm{odd}(8)} < 4$.

\begin{enumerate}
\item[(2-i)] Suppose $a \equiv b \equiv 2 \pmod{4}$. Then either $a+b \equiv 0 \pmod{8}$ or $a \equiv b \pmod{8}$, and in either case $\#Q_{\mathrm{odd}(8)} 
\leq 3$.

\item[(2-ii)] Suppose $a \equiv 2 \pmod{4}$ and $b \equiv 0 \pmod{8}$. Since $b \equiv 0 \pmod{8}$, we have $\#Q_{\mathrm{odd}(8)} \leq 2$.
\end{enumerate}
Thus $Q(x,y,z)$ has an obstruction modulo $8$.

\item[(3)] Suppose $a \equiv 2 \pmod{4}$ and $b \equiv 4 \pmod{8}$. Then
\[
ax^2 \equiv 0, a, 8 \pmod{16}, \qquad by^2 \equiv 0, b \pmod{16}, \qquad cz^2 \equiv 0, c, 4c, c+8 \pmod{16}.
\]
Modulo $8$, since $a \equiv 2 \pmod 4$ and $b \equiv 4 \pmod 8$, we have
$Q_{\mathrm{odd}(8)} = \{c, a+c, b+c, a+b+c\} = \{c,c+2,c+4,c+6\} \pmod 8$ and 
$Q_{\mathrm{even}(8)} = \{0,a, 4,a+4 \} = \{0,2,4,6\} \pmod 8$.
Thus, $Q(x,y,z)$ has no obstruction modulo $8$.
We therefore turn to the residues $\equiv 2 \pmod{4}$ modulo $16$, and obtain
\[
Q_{2(4) \pmod{16}} = \{\, a,\ a+b,\ a+4c,\ a+b+4c \,\}.
\]
Hence, to complete the proof it suffices to show $\#Q_{2(4)\pmod{16}} < 4$. Note that $4c \equiv 4, 12 \pmod{16}$ and $b \equiv 4, 12 \pmod{16}$.

\begin{enumerate}
\item[(3-i)] Suppose $b \equiv 4c \pmod{16}$. Then $a+b \equiv a+4c \pmod{16}$, so $\#Q_{2(4)\pmod{16}} \leq 3$.

\item[(3-ii)] Suppose $b \not\equiv 4c \pmod{16}$. Then $b+4c \equiv 0 \pmod{16}$, so $\#Q_{2(4)\pmod{16}} \leq 3$.
\end{enumerate}
Therefore $Q(x,y,z)$ has an obstruction modulo $16$.
\end{enumerate}
\end{proof}

\subsection{(iii) Exactly one of \texorpdfstring{$a,b,c$}{a, b, c} is even}
As for $\epsilon_{Q,p}$, we may assume that $a$ is even and $b,c$ are odd.
\begin{prop}\label{prop:ternary-one-even}
Let $Q(x,y,z)=ax^2+by^2+cz^2$ with $a$ even and $b,c$ odd. Then
\[
   \epsilon_{Q,2} =
   \begin{cases}
      2      & \text{if } 4 \mid a, \\
      \infty      & \text{if } a \equiv 2 \pmod 4 \text{ and } a+b+c \equiv 0 \pmod 8 \text{, or }\\
                  & a \equiv 2 \pmod 4 \text{ and } b+c \equiv 0 \pmod 8, \\
      4      &\text{(otherwise)}.
   \end{cases}
\]
\end{prop}

\begin{lem}\label{lem:32-adic-lift}
Let $Q(x,y,z) = ax^2 + by^2 + cz^2$, where $a \equiv 2 \pmod4$ and $b,c$ are odd.
Suppose that $Q(x_5,y_5,z_5) \equiv n_5 \pmod{32}$ and that $(x_5,y_5,z_5)$ satisfies at least one of the following conditions: $x_5$ is odd, $y_5 \equiv 2 \pmod{4}$, or $z_5 \equiv 2 \pmod{4}$.
Then, for every $k \geq 6$, every integer $n$ with $n \equiv n_5 \pmod {32}$ is represented by $Q$ modulo $2^k$.
\end{lem}

\begin{proof}[Proof of \cref{lem:32-adic-lift}]
Fix an integer $n\equiv n_5\pmod{32}$. We proceed by induction on $k$.

Suppose $Q(x_{k-1},y_{k-1},z_{k-1}) \equiv n \pmod{2^{k-1}}$ and that $(x_{k-1},y_{k-1},z_{k-1})$
satisfies the assumption above.

Define
\[
x_k = x_{k-1} + 2^{k-3}u, \qquad y_k = y_{k-1} + 2^{k-3}v, \qquad z_k = z_{k-1} + 2^{k-3}w.
\]
Then
\[
Q(x_k,y_k,z_k) \equiv Q(x_{k-1},y_{k-1},z_{k-1}) + 2^{k-2}\bigl(ax_{k-1}u + by_{k-1}v + cz_{k-1}w\bigr) + 2^{2k-6}(\dots) \pmod{2^k}.
\]
Since $2k-6 \geq k$, this simplifies to
\[
Q(x_k,y_k,z_k) \equiv Q(x_{k-1},y_{k-1},z_{k-1}) + 2^{k-2}\bigl(ax_{k-1}u + by_{k-1}v + cz_{k-1}w\bigr) \pmod{2^k}.
\]
By the hypothesis on $(x_{k-1},y_{k-1},z_{k-1})$, we can choose $u,v,w$ so that
\[
ax_{k-1}u + by_{k-1}v + cz_{k-1}w \equiv 0 \text{ or } 2 \pmod{4},
\]
and hence choose the next binary digit so that $Q(x_k,y_k,z_k)\equiv n\pmod{2^k}$. The increments are divisible by $8$, so the hypothesis on the coordinates persists. This completes the induction.
\end{proof}

\begin{proof}[Proof of \cref{prop:ternary-one-even}]
\begin{enumerate}
\item[(1)] Suppose $4 \mid a$. Then $Q(x,y,z) \equiv \pm y^2 \pm z^2 \pmod{4}$.
A direct check shows that none of the four possible forms $\pm y^2\pm z^2$ represents every residue class modulo $4$.

\item[(2)] Suppose $a\equiv2\pmod4$ and either $a+b+c \equiv 0 \pmod{8}$ or $b+c \equiv 0 \pmod{8}$. We prove this case by combining \cref{lem:8-adic-lift,lem:32-adic-lift}.

For $(x,y,z) = (0,1,0),\,(1,1,0),\,(0,1,2),\,(1,1,2) \pmod{8}$, the corresponding values of $Q(x,y,z)$ modulo $8$ are, respectively,
\[
b,\quad b+a,\quad b+4,\quad b+a+4.
\]
Since $a \equiv 2 \pmod{4}$, we have $\{b,\ b+a,\ b+4,\ b+a+4\} = \{1,3,5,7\} \pmod{8}$.
Thus, by \cref{lem:8-adic-lift}, $Q(x,y,z)$ represents every odd residue class modulo $2^k$ for every $k$.
We now consider representations of even numbers.

\begin{enumerate}
\item[(2-i)] Suppose $b+c \equiv 0 \pmod{8}$. Consider the set $\{b+c,\ a+b+c,\ 4b,\ a+4b\} \pmod{8}$.
Since $b+c \equiv 0 \pmod{8}$, $a \equiv 2 \pmod{4}$, and $b$ is odd, we have
\[
\{b+c,\ a+b+c,\ 4b,\ a+4b\} = \{0,2,4,6\} \pmod{8}.
\]
Moreover, $(0,1,1)$ and $(1,1,1)$ represent $b+c$ and $a+b+c$, respectively, modulo $8$.
Hence \cref{lem:8-adic-lift} applies, and $Q(x,y,z)$ represents every residue class congruent to $0$ or $a$ modulo $8$, modulo $2^k$ for every $k$.

To lift the solutions with $Q(x,y,z) \equiv 4b, a+4b \pmod{8}$, we gather more information modulo $32$.
Since $a \equiv 2 \pmod{4}$ and $b$ is odd, we have $4a \equiv 8, 24 \pmod{32}$ and $4b \equiv 4, 12, 20, 28 \pmod{32}$.

Thus, for $(x,y,z) = (0,2,0),\,(2,2,0),\,(0,2,4),\,(2,2,4) \pmod{32}$, the corresponding values of $Q(x,y,z)$ modulo $32$ are
\begin{equation}\label{eq:four-mod-eight-lifts}
4b,\quad 4a+4b,\quad 4b+16,\quad 4a+4b+16, \quad \text{i.e.,} \quad \{4,12,20,28\} \pmod{32}.
\end{equation}
Since these solutions satisfy the assumption of \cref{lem:32-adic-lift}, $Q(x,y,z)$ represents every residue class congruent to $4$ modulo $8$, modulo $2^k$ for every $k$.

On the other hand, $16b \equiv 16c \equiv 16 \pmod{32}$, and $4(b+c) \equiv 0 \pmod{32}$ implies $4b \equiv -4c \pmod{32}$.
Thus, for $(x,y,z) = (1,2,0),\,(1,2,4),\,(1,0,2),\,(1,4,2) \pmod{32}$, the corresponding values of $Q(x,y,z)$ modulo $32$ are
\[
a+4b,\quad a+4b+16,\quad a-4b,\quad a-4b+16.
\]
Since $a-4b \equiv a+4b-8b \pmod{32}$ and $-8b \equiv 8, 24 \pmod{32}$, we have
\[
\{a+4b,\ a+4b+16,\ a-4b,\ a-4b+16\} = \{a+4b,a+4b+8,a+4b+16,a+4b+24\} \pmod{32}.
\]
Since these solutions satisfy the assumption of \cref{lem:32-adic-lift}, $Q(x,y,z)$ represents every residue class congruent to $a+4$ modulo $8$, modulo $2^k$ for every $k$.
Hence $Q(x,y,z)$ has no obstruction modulo $2^k$ for any $k$.

\item[(2-ii)] Suppose $a+b+c \equiv 0 \pmod{8}$. The solutions $(1,1,1)$ and $(0,1,1)$ represent $a+b+c\equiv0$ and $b+c\equiv-a\pmod8$, respectively, and satisfy the hypothesis of \cref{lem:8-adic-lift}. Moreover, the four solutions used in \cref{eq:four-mod-eight-lifts} represent all four residue classes modulo $32$ that are congruent to $4$ modulo $8$. Thus, $Q(x,y,z)$ represents every residue class congruent to $0$, $-a$, or $4$ modulo $8$, modulo $2^k$ for every $k$.

Moreover, for $(x,y,z) = (1,0,0),\,(1,4,0),\,(1,2,2),\,(3,2,2) \pmod{32}$, the corresponding values of $Q(x,y,z)$ modulo $32$ are
\[
a,\quad a+16,\quad a+4b+4c,\quad 9a+4b+4c.
\]
Since $a+b+c \equiv 0 \pmod{8}$, we have $4(a+b+c) \equiv 0 \pmod{32}$, while $4a \equiv 8$ or $24 \pmod{32}$. Hence
\[
9a+4b+4c \equiv a+4a+4(a+b+c) \equiv a+8 \text{ or } a+24 \pmod{32},
\]
and correspondingly
\[
a+4b+4c \equiv a - 4a + 4(a+b+c) \equiv a+24 \text{ or } a+8 \pmod{32}.
\]
Since these solutions satisfy the assumption of \cref{lem:32-adic-lift}, they lift to all four residue classes modulo $2^k$ that are congruent to $a$ modulo $8$.
Therefore $Q(x,y,z)$ has no obstruction modulo $2^k$ for any $k$.
\end{enumerate}

\item[(3)] Finally, suppose $a\equiv2\pmod4$, $a+b+c\not\equiv0\pmod8$, and $b+c\not\equiv0\pmod8$. We have
\[
ax^2 \equiv 0, a, 8 \pmod{16}, \qquad by^2 \equiv 0, b, 4b, b+8 \pmod{16}, \qquad cz^2 \equiv 0, c, 4c, c+8 \pmod{16}.
\]
The residues
\[
\{b,b+a,b+4,b+a+4\}=\{1,3,5,7\}\pmod8
\]
are represented, as are
\[
\{0,4,a,a+4\}=\{0,2,4,6\}\pmod8.
\]
Thus, $Q(x,y,z)$ has no obstruction modulo $8$.
We therefore consider representations $\equiv 2 \pmod{4}$ modulo $16$, and show that $\#Q_{2(4)\pmod{16}} < 4$.

\begin{enumerate}
\item[(3-i)] Suppose $b \not\equiv c \pmod{4}$. Since $b+c \not\equiv 2 \pmod{4}$, we have
\[
Q_{2(4)\pmod{16}} = \{a,\ a+4b,\ a+4c,\ a+4b+4c\}.
\]
But $4b+4c \equiv 0 \pmod{16}$, so $\#Q_{2(4)\pmod{16}} \leq 3$.

\item[(3-ii)] Suppose $b \equiv c \pmod{4}$. Since $b+c \equiv 2 \pmod{4}$, we have
\[
Q_{2(4)\pmod{16}} = \{a,\ a+4b,\ a+8,\ b+c,\ b+c+8\}.
\]
If $4b \equiv 4 \pmod{16}$, then $Q_{2(4)\pmod{16}} = \{a,\ a+4,\ a+8,\ b+c,\ b+c+8\}$. We claim that $a+12 \not\equiv b+c,\ b+c+8 \pmod{16}$.
Indeed, either congruence is equivalent to $2a+12 \equiv a+b+c$ or $a+b+c+8 \pmod{16}$; since $2a+12 \equiv 0, 8 \pmod{16}$, this would force $a+b+c \equiv 0 \pmod{8}$, a contradiction.
If $4b \equiv 12 \pmod{16}$, then
$Q_{2(4)\pmod{16}} = \{a,\ a+12,\ a+8,\ b+c,\ b+c+8\}$. In this case the same argument, with $2a+4\equiv0$ or $8\pmod{16}$, shows that $a+4$ is not represented.
Hence at least one of the four residue classes congruent to $2$ modulo $4$ is not represented, so $\#Q_{2(4)\pmod{16}} \leq 3$.
\end{enumerate}
\end{enumerate}
\end{proof}

\subsection{Examples}
Using \cref{main-thm2}, we obtain the following examples.
\begin{ex}
\begin{enumerate}
\item[] \vspace{\baselineskip}
  \item[(1a)] $a\equiv b\equiv c\pmod 4$: \quad
    $Q=5x^2+65y^2+9z^2$, \; $p_{\min}=5$
    $\;\Rightarrow\; \kappa_Q=5$.

  \item[(1b)] otherwise: \quad
    $Q=x^2+3y^2+5z^2$, \; $p_{\min}=\infty$, \; $q_{\min}=5$
    $\;\Rightarrow\; \kappa_Q=25$.
\end{enumerate}

\begin{enumerate}
  \item[(2a)] $4\mid a,b$: \quad
    $Q=12x^2+36y^2+z^2$, \; $p_{\min}=3$
    $\;\Rightarrow\; \kappa_Q=3$.

  \item[(2b)] $a\equiv b\equiv 2\pmod 4$: \quad
    $Q=10x^2+30y^2+z^2$, \; $p_{\min}=5$
    $\;\Rightarrow\; \kappa_Q=5$.

  \item[(2c)] $a\equiv 2\pmod4,\ b\equiv 0\pmod8$: \quad
    $Q=14x^2+56y^2+z^2$, \; $p_{\min}=7$
    $\;\Rightarrow\; \kappa_Q=7$.

  \item[(2d)] $a\equiv 2\pmod4,\ b\equiv 4\pmod8$: \quad
    $Q=2x^2+20y^2+3z^2$, \; $p_{\min}=\infty$, \; $q_{\min}=3$
    $\;\Rightarrow\; \kappa_Q=9$.
\end{enumerate}

\begin{enumerate}
  \item[(3a)] $4\mid a$: \quad
    $Q=28x^2+21y^2+5z^2$, \; $p_{\min}=7$
    $\;\Rightarrow\; \kappa_Q=4$.

  \item[(3b)] $a\equiv 2\pmod4$ and $b+c\equiv0\pmod8$: \quad
    $Q=2x^2+y^2+7z^2$, \; $p_{\min}=\infty$, \; $q_{\min}=7$
    $\;\Rightarrow\; \kappa_Q=49$.

  \item[(3c)] otherwise: \quad
    $Q=2x^2+5y^2+45z^2$, \; $p_{\min}=5$, \; $q_{\min}=3$
    $\;\Rightarrow\; \kappa_Q=5$.
\end{enumerate}
\end{ex}

\section*{Acknowledgements}
The author would like to thank his supervisor, Professor Henrik Bachmann, for his helpful suggestions and continued support.

\bibliographystyle{alphaurl}
\bibliography{reference-LO}

\end{document}